\documentclass[a4paper]{amsart}
\usepackage{amssymb}

\input xy
\xyoption{all}

\newcommand{\B}{\mathcal B}

\newcommand{\R}{\mathbb R}

\newcommand{\N}{\mathbb N}
\newcommand{\T}{\mathcal T}
\newcommand{\F}{\mathcal F}

\newtheorem{theorem}{Theorem}
\newtheorem{df}{Definition}
\newtheorem{lemma}{Lemma}

\newtheorem{problem}{Problem}

\begin{document}

\title{Some remarks on characterization of t-normed integrals on compacta}


\author{Taras Radul}

\maketitle

Institute of Mathematics, Casimirus the Great University of Bydgoszcz, Poland;
\newline
Department of Mechanics and Mathematics, Ivan Franko National University of Lviv,
Universytettska st., 1. 79000 Lviv, Ukraine.
\newline
e-mail: tarasradul@yahoo.co.uk

\textbf{Key words and phrases:}  Capacity,   t-normed integral

\subjclass[MSC 2020]{ 28E10}

\begin{abstract} A characterization of t-normed integrals was obtained in   \cite{CLM} for finite compacta and in \cite{Rad} for the general case. Such characterization establishes a correspondence between the space of  capacities and homogeneous respect t-norm monotone normed functionals preserving the maximum operation of comonotone functions. In fact these theorems we can consider as non-additive and non-linear analogues of well-known  Riesz Theorem about a correspondence between the set of $\sigma$-additive regular Borel measures and the set of linear  positively defined functionals. We discuss optimality of such characterization.
 \end{abstract}

\maketitle

\section{Introduction}

Capacities (non-additive measures, fuzzy measures) were introduced by Choquet in \cite{Ch} as a natural generalization of additive measures. They found numerous applications (see for example \cite{EK},\cite{Gil},\cite{Sch}). Capacities on compacta were considered in \cite{Lin} where the important role plays the upper-semicontinuity property which  connects the capacity theory with the topological structure. Categorical and topological properties of spaces of upper-semicontinuous normed capacities on compact Hausdorff spaces were investigated in \cite{NZ}.

In fact, the most of applications of non-additive measures to game theory, decision making theory, economics etc deal not with measures as set functions  but with integrals which allow to obtain expected utility or expected pay-off.  Several types of integrals with respect to non-additive measures were developed for different purposes (see for example \cite{Grab}, \cite{KM}, \cite{LMOS}, and  \cite{Den}). Such integrals are called fuzzy integrals. The most known are the Choquet integral based on the addition and the mninimum operations \cite{Ch} and the Sugeno integral  based on the maximum and the minimum operations \cite{Su}. If we change the minimum operation by any t-norm, we obtain the  generalization of the Sugeno integral called t-normed integrals \cite{Sua}. One of the important problems of the fuzzy integrals theory is characterization of integrals as functionals on some function space (see for example subchapter 4.8 in \cite{Grab} devoted to characterizations of the Choquet integral and the Sugeno integral). A characterization of t-normed integrals was obtained in   \cite{CLM} for finite compacta and in \cite{Rad} for the general case.

It was remarked in \cite{Grab} that in particular case, when we consider the Sugeno integral on finite sets, some of conditions used in the characterization theorem in \cite{CLM} are superfluous and   a simpler characterization of the Sugeno integral on finite sets was given. Another simpler characterization of the Sugeno integral  was obtained in \cite{Ma} for finite sets and  \cite{Nyk} or \cite{R} for any compactum.

The main aim of this paper is to analyze if we can generalize one of mentioned simplification to the general case of t-normed integral changing the minimum operation by any t-norm. We will answer this question in negative. We obtain in Section 2 a characterization of the Sugeno integral which is simpler then one in \cite{Ma}. But  we show in Section 3 that the characterization from \cite{Ma} can not be generalized for any t-normed integral even for the finite case. However we generalize the characterization of Sugeno integral given in \cite{Grab} for any t-normed integral and any compactum in Section 4. Finally, we discuss an open problem in Section 5.

\section {Capacities   and t-normed integrals}

In what follows, all spaces are assumed to be compacta (compact Hausdorff space) except for $\R$ and the spaces of continuous functions on a compactum. All maps are assumed to be continuous. By $\F(X)$ we denote the family of all closed subsets of a compactum $X$.

We shall denote the
Banach space of continuous functions on a compactum  $X$ endowed with the sup-norm by $C(X)$. For any $c\in\R$ we shall denote the
constant function on $X$ taking the value $c$ by $c_X$. We also consider the natural lattice operations $\vee$ and $\wedge$ ( on $C(X)$ and  its sublattices $C(X,[0,+\infty))$ and $C(X,[0,1])$.

We need the definition of capacity on a compactum $X$. We follow a terminology of \cite{NZ}.

\begin{df} A function $\nu:\F(X)\to [0,1]$  is called an {\it upper-semicontinuous capacity} on $X$ if the three following properties hold for each closed subsets $F$ and $G$ of $X$:

1. $\nu(X)=1$, $\nu(\emptyset)=0$,

2. if $F\subset G$, then $\nu(F)\le \nu(G)$,

3. if $\nu(F)<a$ for $a\in[0,1]$, then there exists an open set $O\supset F$ such that $\nu(B)<a$ for each compactum $B\subset O$.
\end{df}

By $MX$ we denote the space $MX$ of all upper-semicontinuous  capacities on a compactum $X$

 Remind that a triangular norm $\ast$ is a binary operation on the closed unit interval $[0,1]$ which is associative, commutative, monotone and $s\ast 1=s$ for each  $s\in [0,1]$ \cite{PRP}. Let us remark that the monotonicity of $\ast$ implies distributivity, i.e. $(t\vee s)\ast l=(t\ast l)\vee (s\ast l)$ for each $t$, $s$, $l\in[0,1]$.  We consider only continuous t-norms in this paper.

Integrals generated  by  t-norms are called t-normed integrals and were studied in \cite{We1}, \cite{We2} and \cite{Sua}. Denote $\varphi_t=\varphi^{-1}([t,1])$ for each $\varphi\in C(X,[0,1])$ and $t\in[0,1]$. So, for a continuous t-norm $\ast$, a capacity $\mu$ and a  function $f\in  C(X,[0,1])$ the corresponding t-normed integral is defined by the formula $$\int_X^{\vee\ast} fd\mu=\max\{\mu(f_t)\ast t\mid t\in[0,1]\}.$$

Let $X$ be a compactum.  We call two functions $\varphi$, $\psi\in C(X,[0,1])$ comonotone (or equiordered) if $(\varphi(x_1)-\varphi(x_2))\cdot(\psi(x_1)-\psi(x_2))\ge 0$ for each $x_1$, $x_2\in X$. Let us remark that a constant function is comonotone to any function $\psi\in C(X,[0,1])$.

Let $\ast$ be a continuous t-norm. We denote for a compactum $X$ by $\T^\ast(X)$ the set of functionals $\mu:C(X,[0,1])\to[0,1]$ which satisfy the conditions:

\begin{enumerate}
\item $\mu(1_X)=1$ ($\mu$ is normed);
\item $\mu(\varphi)\le\mu(\psi)$ for each functions $\varphi$, $\psi\in C(X,[0,1])$ such that $\varphi\le\psi$ ($\mu$ is monotone) ;
\item $\mu(\psi\vee\varphi)=\mu(\psi)\vee\mu(\varphi)$ for each comonotone functions $\varphi$, $\psi\in C(X,[0,1])$ ($\mu$ preserves $\vee$ for  comonotone functions);
\item $\mu(c_X\ast\varphi)=c\ast\mu(\varphi)$ for each $c\in[0,1]$ and $\varphi\in C(X,[0,1])$ ($\mu$ is $\ast$-homogeneous).

\end{enumerate}

It was proved in \cite{CLM} for finite compacta and in \cite{Rad} for the general case  that  a   functional $\mu$ on  $C(X,[0,1])$ belongs to $\T^\ast(X)$ if and only if there exists a unique capacity $\nu$ such that $\mu$ is the t-normed integral with respect to $\nu$.

Let us remark that the above characterization can be simplified in the particular case for the Sugeno integral (when $\ast=\wedge$). We can replace Property 3 by a weaker condition:  $\mu(c_X\vee\varphi)=c\vee\mu(\varphi)$ for each $c\in[0,1]$ and $\varphi\in C(X,[0,1])$ (see \cite{Ma} for finite sets and  \cite{Nyk} for any compactum. See also \cite{R} where  some modification of  the Sugeno integral was considered).

The following lemma implies another simplification of this characterization of the Sugeno integral.

\begin{lemma}\label{mon} Let $X$ be a compactum and $\mu:C(X)\to [0,1]$ be a  $\vee$-homogeneous functional. Then $\mu$ is normed. If $\mu$ is additionally  $\wedge$-homogeneuos, then it is  monotone.
\end{lemma}

\begin{proof} Let $\mu:C(X)\to [0,1]$ be a  $\vee$-homogeneous functional. We have $\mu(1_X)=\mu(1_X\vee 1_X)=1\vee\mu(1_X)=1$. Hence $\mu$ is normed.

Now, let $\mu:C(X)\to [0,1]$ be a  $\vee$-homogeneous and $\wedge$-homogeneous functional.
Take any $\varphi$, $\psi\in C(X,[0,1])$ such that $\psi\le\varphi$. Suppose the contrary $\mu(\varphi)=b<a=\mu(\psi)$. Choose $c,$ $d\in [0,1]$ such that $b<c<d<a$.

Since $\psi\le\varphi$, we have $$\varphi^{-1}([0,c])\cap\psi^{-1}([d,1])=\emptyset.$$ Choose a function $\xi\in C(X,[0,1])$ such that $$\xi|_{\varphi^{-1}([0,c])}=\varphi|_{\varphi^{-1}([0,c])}\text{, }\xi|_{\psi^{-1}([d,1])}=\psi|_{\psi^{-1}([d,1])}$$ and $$\xi(X\setminus(\varphi^{-1}([0,c])\cup\psi^{-1}([d,1])))\subset [c,d].$$
Then we have $$\xi\wedge c_X=\varphi\wedge c_X\text{ and }\xi\vee d_X=\psi\vee d_X.$$
Since $\mu$ is $\wedge$-homogeneous, we have  $$\mu(\xi)\wedge c=\mu(\xi\wedge c_X)=\mu(\varphi\wedge c_X)=\mu(\varphi)\wedge c=b\wedge c=b,$$
what implies $\mu(\xi)=b.$

On the other hand, using $\vee$-homogeneity of $\mu$ we obtain $$\mu(\xi)\vee d=\mu(\xi\vee d_X)=\mu(\psi\vee d_X)=\mu(\psi)\vee d=a\vee d=a,$$
what implies $\mu(\xi)=a.$ We have a contradiction.
\end{proof}

So, we obtain a further simplification of characterization of the Sugeno integral.

\begin{theorem}\label{?} A   functional $\mu:C(X,[0,1])\to[0,1]$ is $\vee$-homogeneous and $\wedge$-homogeneous if and only if there exists a unique capacity $\nu$ such that $\mu$ is the Sugeno integral with respect to $\nu$.
\end{theorem}

We will show that we can not generalize the above theorem for any t-norm, moreover we can not even generalize above mentioned results from \cite{Ma}  and  \cite{Nyk} for any t-norm. Let us formulate it more precisely. Let $\ast$ be a continuous t-norm.
Denote for a compactum $X$ by $\T_1^\ast(X)$ the set of  monotone $\vee$-homogeneous and $\ast$-homogeneous  functionals $\mu:C(X,[0,1])\to[0,1]$. The following question arises naturally: if $\mu\in\T_1^\ast(X)$ implies that there exists a unique capacity $\nu$ such that $\mu$ is the t-normed integral with respect to $\nu$? We answer this question in negative in the next section  building a functional $\mu\in\T_1^\cdot(\{1,2,3\})\setminus\T^\cdot(\{1,2,3\})$, where $\cdot$ is usual multiplication on $[0,1]$.

\section{An example of  monotone $\cdot$-homogeneuos and $\vee$-homogeneous functional which does not preserve $\vee$ for  comonotone functions.}

The main aim of this section
is to build a functional $\mu\in\T_1^\cdot(\{1,2,3\})\setminus\T^\cdot(\{1,2,3\})$, where $\cdot$ is usual multiplication on $[0,1]$.

We will develop some technique of extending of  monotone $\cdot$-homogeneuos and $\vee$-homogeneous functionals defined on some subsets of $C(X,[0,1])$ for a compactum $X$. By $C_0$ we denote the subset of $C(X,[0,1])$ consisting of all constant functions.

\begin{df} A subset $A\subset C(X,[0,1])$ is called a $(\vee,\cdot)$-subspace of  $C(X,[0,1])$ if $C_0\subset A$ and $c_X\vee \psi\in A$, $c_X\cdot \psi\in A$ for each $c\in [0,1]$ and $\psi\in A$.
\end{df}

We will use simpler denotations $c\vee \psi$, $c\cdot \psi$ instead $c_X\vee \psi$, $c_X\cdot \psi$ and $c\le(\ge)\psi$ instead $c_X\le(\ge)\psi$ in the following.

The following theorem can be considered as partial idempotent version of Hahn-Banach Theorem and the proof uses the main idea of the proof of Hahn-Banach Theorem.

\begin{theorem}\label{hb1} Let $\mu:A\to [0,1]$ be a  monotone $\cdot$-homogeneuos and $\vee$-homogeneous functional where $A$ is a $(\vee,\cdot)$-subspace of  $C(X,[0,1])$. Then there exists a  monotone $\cdot$-homogeneuos and $\vee$-homogeneous functional $\mu_1:C(X,[0,1])\to [0,1]$ such that $\mu_1|A=\mu$.
\end{theorem}

\begin{proof}
Consider any $(\vee,\cdot)$-subspace $A$ of  $C(X,[0,1])$ and a  monotone $\cdot$-homogeneuos and $\vee$-homogeneous functional $\mu:A\to [0,1]$. Let $\varphi\in C(X,[0,1])$. Put $A'=\{d\vee(c\cdot\varphi)\mid d,$ $c\in[0,1]\}\cup A$. It is easy to check that $A'$ is a $(\vee,\cdot)$-subspace of  $C(X,[0,1])$. We will build a  monotone $\cdot$-homogeneuos and $\vee$-homogeneous functional $\mu':A'\to [0,1]$ which is an extension of $\mu$.

Define a subset of $\R$    as follows  $$H=\{c^{-1}\cdot\mu(\psi)\mid c\in(0,1] \text{ and } \psi \in A \text{ such that } \psi \ge c\cdot \varphi\}.$$ Put $a=\inf H$. It is easy to check that  $a\in[0,1]$.

Define a functional $\mu':A'\to[0,1]$  as follows $\mu'(\psi)= \mu(\psi)$ for each $\psi\in A$ and $\mu'(d\vee(c\cdot\varphi))= d\vee(c\cdot a)$ for $d,$ $c\in[0,1]$. It is easy to see that $\mu'$ is a  $\cdot$-homogeneuos and $\vee$-homogeneous functional.

Let us show that $\mu'$ is monotone. It is obvious for a pair of functions from $A$. It is easy to check for a pair of functions from $\{d\vee(c\cdot\varphi)\mid d,$ $c\in[0,1]\}$.

Now, consider any $\psi\in A$ and $d,$ $c\in[0,1]$ such that $d\vee(c\cdot\varphi)\le\psi$.
We can assume that $c>0$.  We have $d\le\psi$ and $c\cdot\varphi\le \psi$.   We also have $c^{-1}\cdot\mu(\psi)\ge a$ by the choice of $a$. Since $\mu$ is monotone on $A$ and $C_0\subset A$, we have $\mu(\psi)\ge d$. Hence $\mu'(\psi)=\mu(\psi)\ge d\vee(c\cdot a)=\mu'(d\vee(c\cdot\varphi))$.

Conversely, consider  any $\psi\in A$ and $d,$ $c\in[0,1]$ such that $d\vee(c\cdot\varphi)\ge\psi$. We  can assume that $c>0$. Suppose the contrary $d\vee(c\cdot a)<\mu(\psi)$.  Since $a=\inf\{s^{-1}\cdot\mu(\xi)\mid s\in(0,1] \text{ and } \xi \in A \text{ such that } \xi \ge s\cdot \varphi\}<c^{-1}\mu(\psi)$, we can choose $\xi\in A$ such that $\xi\ge s\cdot \varphi$ and $s^{-1}\cdot\mu(\xi)<c^{-1}\cdot\mu(\psi)$.

If $s\le c$ we put $\psi'=s\cdot c^{-1}\cdot\psi\in A$. Then we have $\mu(\psi')=s\cdot c^{-1}\cdot\mu(\psi)$ and $s^{-1}\cdot\mu(\xi)<s^{-1}\cdot\mu(\psi')$. Hence $\mu(\xi)<\mu(\psi')$. Since $\xi\in A$, we have $\mu(s\cdot c^{-1}\cdot d\vee\xi)=s\cdot c^{-1}\cdot d\vee\mu(\xi)<\mu(\psi')$.  On the other hand $\psi'\le s\cdot c^{-1}\cdot d$ and $\psi'\le s\cdot\varphi\le \xi$. Hence $\psi'\le s\cdot c^{-1}\cdot  d\vee\xi$ and we obtain a contradiction with the monotonicity of $\mu$.

Now, let us consider $s>c$. Put $\xi'=c\cdot s^{-1}\cdot\xi\in A$. We have $c^{-1}\cdot\mu(\xi')=s^{-1}\cdot\mu(\xi)<c^{-1}\cdot\mu(\psi)$. Hence $\mu(\xi)<\mu(\psi')$ and $\mu(d\vee\xi')=d\vee\mu(\xi')<\mu(\psi)$. On the other hand $\psi\le d\vee c\cdot \varphi= d\vee c\cdot s^{-1}\cdot (s\cdot \varphi)\le d\vee c\cdot s^{-1}\cdot \xi=d\vee  \xi'$ and we obtain a contradiction with monotonicity of $\mu$.

Zorn Lemma implies existence of a  monotone $\cdot$-homogeneuos and $\vee$-homogeneous functional $\mu_1:C(X,[0,1])\to [0,1]$ such that $\mu_1|A=\mu$.
\end{proof}

Now we are able construct the announced example. By $X=\{1,2,3\}$ we denote 3-point space with the discrete topology.  Define two functions $\varphi_1$, $\varphi_2\in C(X,[0,1])=[0,1]^3$ as follows $\varphi_1(1)=0$, $\varphi_1(2)=1/3$, $\varphi_1(3)=2/3$ and $\varphi_2(1)=1/3$, $\varphi_2(2)=1/3$, $\varphi_2(3)=1$. Put $\varphi_3=\varphi_1\vee\varphi_2$. Then we have  $\varphi_3(1)=1/3$, $\varphi_3(2)=1/2$ and  $\varphi_3(3)=1$. Let us remark that $\varphi_1$ and $\varphi_2$ are comonotone.

For $i\in\{1,2,3\}$ define sets $H_i\subset C(X,[0,1])$ as follows $H_i=\{d\vee(c\cdot\varphi_i)\mid d,$ $c\in[0,1]\}$ and put $H=\cup_{i=1}^3H_i$. It is easy to check that $H$ is a $(\vee,\cdot)$-subspace of  $C(X,[0,1])$.

Put $m_1=m_2=1/3$ and $m_3=1/2$. Define a functional $\mu:H\to [0,1]$ by the formula $\mu(d\vee(c\cdot\varphi_i))=d\vee(c\cdot m_i)$ for $i\in\{1,2,3\}$ and $d,$ $c\in[0,1]\}$. It is easy to see that $\mu$ is   $\cdot$-homogeneuos and $\vee$-homogeneous. Let us show that $\mu$ is monotone. It is a routine checking for a pair of functions belonging to the same  $H_i$ and we omit is. Let us check it for functions from distinct $H_i$.

Let $\psi_1\in H_1$ and $\psi_2\in H_2$. Then we have $$\psi_1=d_1\vee(c_1\cdot\varphi_1),\text{ }  \psi_2=d_2\vee(c_2\cdot\varphi_2)$$ for some $d_1,$  $d_2,$ $c_1,$ $c_2\in[0,1]$ and $$\mu(\psi_1)=d_1\vee(c_1/3),\text{ }\mu(\psi_2)=d_2\vee(c_2/3).$$

Consider the case when $\psi_1\le\psi_2$. In particular, we have $$d_1\vee(c_1/3)\le d_1\vee(c_1/2)=\psi_1(2)\le\psi_2(2)=d_2\vee(c_2/3).$$
We obtain $\mu(\psi_1)\le\mu(\psi_2).$

Conversely, let $\psi_1\ge\psi_2$. In particular, we have $$d_1\vee(c_1/3)\ge d_1=\psi_1(1)\ge\psi_2(1)=d_2\vee(c_2/3).$$
Hence $\mu(\psi_1)\ge\mu(\psi_2).$

Let $\psi_1\in H_1$ and $\psi_3\in H_3$. Then we have $$\psi_1=d_1\vee(c_1\cdot\varphi_1),\text{ }  \psi_3=d_3\vee(c_3\cdot\varphi_3)$$ for some $d_1,$  $d_3,$ $c_1,$ $c_3\in[0,1]$ and $$\mu(\psi_1)=d_1\vee(c_1/3),\text{ }\mu(\psi_3)=d_3\vee(c_3/2).$$

Consider the case when $\psi_1\le\psi_3$. In particular, we have $$d_1\vee(c_1/3)\le d_1\vee(c_1/2)=\psi_1(2)\le\psi_3(3)=d_3\vee(c_3/2).$$
We obtain $\mu(\psi_1)\le\mu(\psi_3).$

Conversely, let $\psi_1\ge\psi_3$. In particular, we have $$d_1=\psi_1(1)\ge\psi_3(1)=d_3\vee(c_3/3).$$ Suppose the contrary $$d_1\vee(c_1/3)<d_3\vee(c_3/2).$$ The above  inequalities imply $$d_1<c_3/2=d_3\vee(c_3/2).$$ On the other hand we have $$d_1\vee(2c_2/3)=\psi_1(3)\ge\psi_3(3)=d_3\vee c_3=c_3.$$ Thus, we obtain  $2c_2/3\ge c_3$ or  $c_2/3\ge c_3/2$ and we have a contradiction. Hence $\mu(\psi_1)\ge\mu(\psi_3).$

Let $\psi_2\in H_2$ and $\psi_3\in H_3$. Then we have $$\psi_2=d_2\vee(c_2\cdot\varphi_2),\text{ }  \psi_3=d_3\vee(c_3\cdot\varphi_3)$$ for some $d_2,$  $d_3,$ $c_2,$ $c_3\in[0,1]\}$ and $$\mu(\psi_2)=d_2\vee(c_2/3),\text{ }\mu(\psi_3)=d_3\vee(c_3/2).$$

Then we have $$\mu(\psi_2)=\psi_2(2),\text{ and }\mu(\psi_3)=\psi_3(2).$$
Hence $\psi_2\le(\ge)\psi_3$ implies $\mu(\psi_2)\le(\ge)\mu(\psi_3)$ and we obtain that $\mu$ is monotone.

Let us remark that our arguments also imply that $\mu$ is correctly defined.

Thus the functional $\mu$  is   $\cdot$-homogeneuos, $\vee$-homogeneous and monotone. But we have $$\mu(\psi_1)\vee\mu(\psi_2)=1/3\ne1/2=\mu(\psi_3)=\mu(\psi_1\vee\psi_2),$$
hence $\mu$ does not preserve $\vee$ for  comonotone functions. By Theorem \ref{hb1} we can extend $\mu:H\to[0,1]$ to a   $\cdot$-homogeneuos, $\vee$-homogeneous and monotone functional $\mu_1:C(X,[0,1])\to[0,1]$ which does not preserve $\vee$ for  comonotone functions.

\section{Functionals which  preserve $\vee$ for  comonotone functions.}

Another characterization of the Sugeno integral on finite sets was given by Theorem 4.58 in \cite{Grab}. Let $X$ be a finite set and $|X|=n\in\N$. The characterization theorem is formulated for the Sugeno integral defined on  functions from $C(X,[0,+\infty))=[0,+\infty)^n$. Reformulating this theorem for functions from $C(X,[0,1])=[0,1]^n$ we obtain that a   functional $\mu:[0,1]^n\to[0,1]$ has the properties

\begin{enumerate}
\item $\mu(1_X)=1$;
\item $\mu$ preserves $\vee$ for  comonotone functions;
\item $\mu(c_X\wedge\chi_A)=c\wedge\mu(\chi_A)$ for each $c\in[0,1]$ and $A\subset X$ (by $\chi_A$ we denote the characteristic function of the set $A$).
\end{enumerate}

if and only if there exists a unique capacity $\nu$ on $X$ such that $\mu$ is the Sugeno integral with respect to $\nu$.

We can  rewrite the proof of Theorem 4.58 from \cite{Grab} to obtain its generalization.

\begin{theorem}\label{finite} Let $X$ be a finite set and $|X|=n\in\N$ and $\ast$ is a continuous t-norm. A   functional $\mu:[0,1]^n\to[0,1]$ has the properties

\begin{enumerate}
\item $\mu(1_X)=1$;
\item $\mu$ preserves $\vee$ for  comonotone functions;
\item $\mu(c_X\ast\chi_A)=c\ast\mu(\chi_A)$ for each $c\in[0,1]$ and $A\subset X$ (by $\chi_A$ we denote the characteristic function of the set $A$).
\end{enumerate}

if and only if there exists a unique capacity $\nu$ on $X$ such that $\mu$ is the t-normed integral with respect to $\nu$.
\end{theorem}

Let us consider the general case of any compactum. Since the characteristic function $\chi_A$ generally speaking is not continuous, we will use the $\ast$-homogeneity of functionals instead the second property from the above theorem.

For $A\in\F(X)$ put $\Upsilon_A=\{\varphi\in C(X,[0,1])\mid \varphi(a)=1$ for each $a\in A\}$. If $A=\emptyset$ we put $\Upsilon_A=C(X,[0,1])$.

\begin{lemma}\label{Comon} Let  $\varphi\in C(X,[0,1])$, $\delta<\xi<1$. Then there exists $\psi\in\Upsilon_{\varphi_\xi}$ such that $\psi|_{\varphi^{-1}[0,\delta]}=\varphi|_{\varphi^{-1}[0,\delta]}$ and $\psi$ is comonotone with $\varphi$.
\end{lemma}

\begin{proof} Consider the linear function $\kappa:[\delta,\xi]\to[1,\xi^{-1}]$ such that $\kappa(\delta)=1$ and $\kappa(\xi)=\xi^{-1}$.

Put
$$ \psi(x)=\begin{cases}
\varphi(x),&\varphi(x)\le \delta,\\
\kappa(\xi)\varphi(x),&\delta\le\varphi(x)\le \xi,\\
1,&\xi\le\varphi(x).\end{cases}$$

It is easy to check that $\psi'$ is a function we are looking for.
\end{proof}

The following characterization theorem is a modification of the above mentioned result from \cite{Rad}. In fact, we will see that the property of monotonicity is obsolete. The proof is also a modification of the proof given in \cite{Rad}.  We simply reduce using monotonicity to the pairs of comonotone functions.

We denote by $\B$ the set of functionals $I:C(X,[0,1])\to[0,1]$ which satisfy the properties

\begin{enumerate}
\item $\mu(1_X)=1$;
\item $I$ preserves $\vee$ for  comonotone functions;
\item $I(c_X\ast\varphi)=c\ast I(\varphi)$ for each $c\in[0,1]$ and $\varphi\in C(X,[0,1])$.
\end{enumerate}

\begin{theorem}\label{general} Let $X$ be a compactum and $\ast$ is a continuous t-norm. A   functional $I:C(X,[0,1])\to[0,1]$ belongs to the set $\B$ if and only if there exists a unique capacity $\nu$ on $X$ such that $I$ is the t-normed integral with respect to $\nu$.
\end{theorem}

\begin{proof} Sufficiency is proved in Lemma 4 from \cite{Rad}.

Necessity. Take any $I\in\B$.  Define $\nu:\F(X)\to [0,1]$ as follows $\nu(A)=\inf\{I(\varphi)\mid \varphi\in \Upsilon_A\}$ if $A\ne\emptyset$ and $\nu(\emptyset)=0$. It is easy to see that $\nu$ satisfies Conditions 1 and 2 from the definition of capacity.

Let $\nu(A)<\eta$ for some $\eta\in [0,1]$ and $A\in\F(X)$. Then there exists $\varphi\in \Upsilon_A$ such that $I(\varphi)<\eta$. Choose $\beta\in [0,1]$ such that $I(\varphi)<\beta<\eta$. Since the operation $\ast$ is continuous,  $\eta\ast 1=\eta$ and $\eta\ast 0=0$, there is $\delta\in [0,1]$ such that $\eta\ast \delta=\beta$. Evidently, $\delta<1$. Choose $\zeta\in [0,1]$ such that $\delta<\zeta<1$. We can choose a function $\psi\in\Upsilon_{\varphi_\zeta}$ comonotone to $\varphi$ such that  $\psi|_{\varphi^{-1}([0,\delta])}=\varphi|_{\varphi^{-1}([0,\delta])}$ by Lemma \ref{Comon}. Then we have $\delta\ast\psi\le\varphi$ and $\delta\ast I(\psi)\le I(\varphi)<\beta=\delta\ast\eta$. Hence $I(\psi)<\eta$. Put $U=\psi^{-1}((\zeta,1])$. Evidently $U$ is open and $U\supset A$. We have $\nu(K)\le I(\psi)<\eta$ for each compactum $K\subset U$. Hence $\nu\in MX$.

Let us show that $\int_X^{\vee\ast} \varphi d\nu=I(\varphi)$ for each $\varphi\in C(X,[0,1])$.  We have $\int_X^{\vee\ast} \varphi d\nu=\max\{\inf\{I(\chi)\mid \chi\in \Upsilon_{\varphi_t}\}\ast t\mid t\in[0,1]\}=\max\{\inf\{I(t\ast\chi)\mid \chi\in \Upsilon_{\varphi_t}\}\mid t\in[0,1]\}$.

The inequality $\inf\{I(\chi)\mid \chi\in \Upsilon_{\varphi_t}\}\ast t\le I(\varphi)$ is obvious for each $t\in[0,I(\varphi)]$. Consider any $t>I(\varphi)$.  By Lemma \ref{Comon} for each $\delta<t$ we can choose a function $\chi^\delta\in\Upsilon_{\varphi_t}$ such that $\chi^\delta|_{\varphi^{-1}([0,\delta])}=\varphi|_{\varphi^{-1}([0,\delta])}$. Then we have $\delta\ast\chi^\delta\le\varphi$ and $\delta\ast I(\chi^\delta)\le I(\varphi)$.  Since the operation $\ast$ is continuous, $\inf\{I(\chi)\mid \chi\in \Upsilon_{\varphi_t}\}\ast t\le I(\varphi)$. Hence $\int_X^{\vee\ast} \varphi d\nu\le I(\varphi)$.

Suppose $b=\int_X^{\vee\ast} \varphi d\nu<I(\varphi)=a$. Put $m=\max_{x\in X}\varphi(x)$.  Then for each $t\in[0,m]$ there exists $\chi^t\in\Upsilon_{\varphi_t}$ such that $I(t\ast\chi^t)<a$.  We can assume that $\chi^t$ is comonotone with $\varphi$ by  Lemma 4 from \cite{Rad}. For each $t\in[0,m)$ choose $t'>t$ such that $t'\ast I(\chi^t)<a$. The set $V_t=\{y\mid t'\ast\chi^t(y)>\varphi(y)\}$ is an open neighborhood for each $x\in X$ with $\varphi(x)=t$.

Now we will choose an open neighborhood $W$ of the set $\varphi_m$.
Put $\xi=\min\{\eta\in[0,1]\mid\eta\ast m\le m\}$. We have $\xi\ast m=m$. Since $a\le m$, using arguments as before we can find $\gamma\in[0,1]$ such that $\gamma\ast m=a$. Then we have $\xi\ast a=\xi\ast \gamma\ast m=\gamma\ast m=a$. Choose $\lambda<\xi$ such that $b<\lambda\ast a$. Then there exists $\omega\in\Upsilon_{\varphi_{m\ast\lambda}}$ such that $m\ast\lambda\ast I(\omega)<\lambda\ast a$, hence  $m\ast I(\omega)< a$. We can assume that $\omega$ is comonotone with $\varphi$ by  Lemma 4 from \cite{Rad}. Since $m\ast\lambda<m$, the open set $W=\varphi^{-1}(m\ast\lambda,m]$ contains the set $\varphi_m$. Let us remark that $m\ast\omega(x)= m$ for each $x\in W$.

  We can choose a finite subcover $\{W,V_{t_1},\dots,V_{t_k}\}$ of the open cover $\{W\}\cup\{V_t\mid t\le m\}$. We can assume that $t_0=m\ast\lambda$ and $t_i\in [0,m)\setminus \{t_0\}$ for $i>0$.
 Then we have that the functions $t_i'\ast\chi^{t_i}$ and $\omega$ are pairwise comonotone, hence $I(\omega\bigvee(\bigvee_{i=1}^k\{t_i'\ast\chi^{t_i}\}))<a$. On the other hand $\varphi\le \omega\bigvee(\bigvee_{i=1}^k\{t_i'\ast\chi^{t_i}\})$ and we obtain a contradiction.
\end{proof}

\section {Monotonicity of functionals:  open problem.}

Following \cite{Grab} we say that a functional  $I:C(X,[0,1])\to[0,1]$ is comonotonically maxitive if it preserves preserves $\vee$ for  comonotone functions. We compare the properties of monotonicity  and comonotonical maxitivity in this section.  Results of Section 3 in particular demonstrate that monotonicity does not imply comonotonical maxitivity. It is easy to see that a comonotonically maxitive functional is monotone for  pairs of comonotone functions.  Theorem \ref{general} with Lemma 4 from \cite{Rad} implies that each normed, $\ast$-homogeneous respect a continuous t-norm and  comonotonically maxitive functional is monotone. The following theorem shows that the  first and the second conditions are superfluous for a finite compactum.

\begin{theorem}\label{finite1} Let $X$ be a finite set and a  functional $\mu:C(X,[0,1])=[0,1]^n\to[0,1]$ is comonotonically maxitive. Then $\mu$ is monotone.
\end{theorem}

\begin{proof} Let $X=\{x_1,\dots,x_n\}$ and $\varphi$, $\psi:X\to [0,1]$ are two functions such that $\psi\le \varphi$. We can assume that that the function $\varphi$ is nondecreasing  (if not, apply some permutation on $X$).

We  build recursively a finite sequence $\psi_1,\dots,\psi_{n-1}$ of functions as follows. Put $\psi_1(x)=\max\{\psi(x),\psi(x_1)\}$. For $i\in\{1,\dots,n-2\}$ we define the function $\psi_{i+1}$ by the formula
$$\psi(x_j)=\begin{cases}
\psi_i(x_j),&j\le i,\\
\max\{\psi_i(x_{i+1}),\psi_i(x_j)\},&j\ge i+1.\end{cases}$$

It is easy to check that $\psi\le\psi_1\le,\dots,\psi_{n-1}\le\varphi$, $\psi$ is comonotone to $\psi_1$,  $\psi_i$ is comonotone to $\psi_{i+1}$ and $\psi_{n-1}$ is nondecreasing, thus comonotone to $\varphi$. Hence we have $\mu(\psi)\le\mu(\psi_1)\le\dots\le \mu(\psi_{n-1})\le\mu(\varphi)$.
\end{proof}

But the problem is still open in the general case.

\begin{problem} Let $X$ be a compactum and a  functional $\mu:C(X,[0,1])\to[0,1]$ is comonotonically maxitive. Is $\mu$  monotone?
\end{problem}

\end{document}